\documentclass{article}
\usepackage[utf8]{inputenc}
\usepackage{float}
  \usepackage{amsmath}
  \usepackage{amsthm}
 \usepackage{amssymb}
\usepackage{amsfonts}
\usepackage{graphics}
\usepackage{hyperref}
 \usepackage[english]{babel}
 \newlength\tindent
\newtheorem{theorem}{Theorem}

\newtheorem{lemma}{Lemma}
\newtheorem{corollary}{Corollary}
\newtheorem{definition}{Definition}

\newtheorem{prop}{Proposition}
\newtheorem{problem}{Problem}

\begin{document}

\title{On additive irreducibility of multiplicative subgroups}
\author{Alexander Kalmynin \footnote{National Research University Higher School of Economics, Moscow, Russia} \\ email: \href{mailto:alkalb1995cd@mail.ru}{alkalb1995cd@mail.ru}}
\date{}
\maketitle
\begin{abstract}
In this paper, we employ a version of Stepanov's method, developed by Hanson and Petridis, to prove several results on additive irreducibility of multiplicative subgroups in finite fields of prime order $p$. Specifically, we show that if a subgroup $\mu_d$ of $d$-th roots of unity satisfies $A-A=\mu_d\cup\{0\}$, then $d=2$ or $6$. Additionally, we resolve the S\'ark\"ozy's conjecture on quadratic residues: for prime $p$, the set $\mathcal R_p$ of quadratic residues modulo $p$ cannot be represented as $A+B$ for $A,B$ with $\min(|A|,|B|)>1$. More generally, we prove that if the set of $d$-th roots of unity $\mu_d$ is represented non-trivially as $A+B$, then the sizes of summands are equal.
\end{abstract}

\section{Introduction and main results}

A well-known meta-conjecture in additive combinatorics asserts that multiplicatively structured sets cannot exhibit significant additive structure, and vice versa. For instance, a celebrated conjecture by S\'ark\"ozy \cite{Sha} states that for a sufficiently large prime number $p$, the set $\mathcal R_p$ of all quadratic residues modulo $p$ cannot be non-trivially represented as $A+B$ for $A,B\subset \mathbb F_p$. Here \emph{non-trivially} means that $|A|,|B|>1$. A variant of this conjecture, studied by Lev and Sonn \cite{LeSo}, states that the set of quadratic residues with $0$ included is not of the form $A-A$ for $p>13$. More precisely, they derived necessary conditions for the equality
\[
A-A\stackrel{!}{=}\mathcal R_p,
\]
where $\stackrel{!}{=}$ indicates that every element $r\in \mathcal R_p$ is uniquely represented as $a_1-a_2$ with $a_1,a_2\in A$ and $a_1-a_2\in \mathcal R_p$ for all distinct $a_1,a_2\in A$. This problem is partly motivated by Shkredov's result \cite{Shk1}, that shows that for $p\neq 3$ $\mathcal R_p$ cannot be represented as $A+A$ for any $A\subset \mathbb F_p$. Furthermore, Shkredov's work shows that if $A+B=\mathcal R_p$ with $|A|,|B|>1$, then $\left(\frac{1}{6}-o(1)\right)\leq |A|,|B|\leq \left(3+o(1)\right)\sqrt{p}$ as $p\to +\infty$. In this paper, we resolve the conjectures of Lev-Sonn and S\'ark\"ozy.

Here, we study polynomials, introduced by Hanson and Petridis \cite{HaPe} to prove a general upper bound for sets with multiplicative restrictions. We denote these polynomials by $HP(x;A,d)$. They depend on the set $A$ under consideration and the order $d$ of the corresponding multiplicative subgroup. Our main results are as follows

\begin{theorem}[Lev-Sonn conjecture]

Suppose that $p$ is a prime number and $d\mid p-1$, $1<d<p-1$. If $\mu_d$ is the set of all $d$-th roots of unity in $\mathbb F_p$ and $A-A=\mu_d\cup\{0\}$ for some $A\subset \mathbb F_p$
then $d=2$ or $6$. Consequently, the Lev-Sonn conjecture is true.
\end{theorem}

We note that the same result holds for the field $\mathbb C$ of complex numbers instead of $\mathbb F_p$.
\begin{prop}
Let $A\subset \mathbb C$ satisfy $A-A=\mu_d\cup\{0\}$, where $\mu_d$ is the set of $d$-th roots of unity. Then $d=2$ or $6$.
\end{prop}
In this case, the proof is much easier due to the existence of Euclidean metric on $\mathbb C$.
\begin{proof}
Since the identity $A-A=\mu_d$ is shift-invariant, i.e. $B=A+t$ for $t\in \mathbb C$ also satisfies $B-B=\mu_d$, we may assume that $0\in A$. Suppose that $|A|\geq 4$. Then there are 3 non-zero elements $a_1,a_2,a_3\in A$. By assumption, $a_i-a_j\in \mu_d$ for $i\neq j$ so $|a_i-a_j|=1$. Thus, $a_1,a_2,a_3$ form an equilateral triangle with side length $1$. Moreover, since $|a_i|=1$ for all $i$, the triangle's center is at $0$. But the circumradius of equilateral triangle with side length $1$ is $\frac{1}{\sqrt{3}}$, a contradiction. Hence, $|A|\leq 3$ and $d=|\mu_d|\leq |A|(|A|-1)\leq 6$. On the other hand, $A-A$ is invariant under the map $x\mapsto -x$, so $-1$ lies in $\mu_d$. This means that $d$ is even, so the only remaining case is $d=4$ and the set $A$ contains $3$ elements. Without loss of generality, we can assume $A=\{0,a_1,a_2\}$, then both $a_1$ and $a_2$ are equal to $\pm 1$ or $\pm i$. If both are imaginary, their difference cannot be real, so $1\not\in A-A$ and vice versa: if both are real, then $i\not\in A$, so we have, say, $a_1=\pm 1$ and $a_2=\pm i$, so $a_1-a_2$ is neither real nor imaginary, so $a_1-a_2\not\in \mu_4$.

The cases $d=2$ and $d=6$ are achieved by sets $A=\{0,1\}$ and $A=\{0,1,-\omega\}$ with $\omega^2+\omega+1=0$.
\end{proof}

The result of Theorem 1 was previously known for small subgroups. Namely, in the paper \cite{Shk2} Shkredov proved that for any $\varepsilon>0$ multiplicative subgroups of order $d\leq p^{4/5-\varepsilon}$ cannot be represented in the form $(A-A)\setminus\{0\}$ for large enough $p$.

Our next theorem shows that if the subgroup $\mu_d$ is additively reducible, then the summands in the representation are of the same size. We call this the ``$\alpha=\beta$ theorem''.

\begin{theorem}[$\alpha=\beta$ theorem]

If $\mu_d=A+B$ with $|A|,|B|>1$, then $|A|=|B|=\sqrt{d}$.    
\end{theorem}

The best previously known bound shows that for $d=\frac{p-1}{2}$ we have $\frac{|B|}{8}<|A|<8|B|$, see \cite{ChXi}. Also, the aforementioned result by Shkredov provides a similar bound with $18$ instead of $8$.

In the paper \cite{Yip1} Yip used Hanson-Petridis polynomials to prove several results on additive irreducibility in arbitrary finite fields (see also \cite{Yip2}). In particular, he proved that for some constant $M$ a proper subgroup $G\subseteq \mathbb F_p^*$ with $|G|>M$ is not a sum $A+B+C$ with $|A|,|B|,|C|>1$. This result follows from Theorem 2
\begin{corollary}
A proper subgroup $\mu_d\subset\mathbb F_p^*$ cannot be expressed as $A+B+C$ with $|A|,|B|,|C|>1$.
\end{corollary}
\begin{proof}
Assume the contrary, then $\mu_d=A+B+C=A+(B+C)=(A+B)+C$. Theorem 2 implies that $|A+B|=|A|=\sqrt{d}$. In particular, if $b_1,b_2$ are distinct elements of $B$, then the sets $A+b_1$ and $A+b_2$ have $\sqrt{d}$ elements and lie inside $A+B$, hence $A+b_1=A+b_2=A+B$. Therefore, $A=A+b_1-b_2$, which is impossible, since $\mathbb F_p$ has no proper shift-invariant subsets.
\end{proof}
Finally, we prove S\'ark\"ozy's conjecture on additive irreducibility of quadratic residues.

\begin{theorem}(S\'ark\"ozy's conjecture)
For all primes $p$, the set $\mathcal R_p=\mu_{(p-1)/2}$ of quadratic residues modulo $p$ is additively irreducible, i.e. $\mathcal R_p\neq A+B$ with $|A|,|B|>1$.
\end{theorem}
Similarly to the Theorem 1, the version of S\'ark\"ozy's conjecture for small subgroups was proved by Shkredov in the paper \cite{Shk3}. He showed that $\mu_d$ with $C\leq d\leq p^{2/3-\varepsilon}$ is not representable as a sumset $A+B$ with $|A|,|B|>1$ for sufficiently large $p$. Here $\varepsilon>0$ is arbitrary and $C>0$ is an absolute constant.

The proofs of Theorem 1 and Theorems 2 and 3 are largely independent but stem from the identity proved in Lemma 4. Section 2 contains basic results on Hanson-Petridis polynomials, Section 3 resolves the Lev-Sonn conjecture, Section 4 proves Theorem 2 and Section 5 addresses S\'ark\"ozy's conjecture.
\section{Hanson-Petridis polynomials}

Let $p$ be a prime, and let $1<d<p-1$ divide $p-1$. As in the introduction, $\mu_d$ denotes the set of all $d$-th roots of unity in $\mathbb F_p$. In this section, we introduce the notion of $d$-criticality for pairs of subsets of $\mathbb F_p$ and establish some useful properties of polynomials $HP(x;A,d)$.

\begin{definition}[Hanson-Petridis polynomial]
Let $A\subseteq \mathbb F_p$ be a set with $|A|=\alpha$, where $\alpha+d-1<p$ and $\alpha>1$. The \textbf{Hanson-Petridis polynomial} for $A$ and $d$ is defined as
\[
HP(x;A,d)=\sum_{a\in A}c_a(A)(x+a)^{d+\alpha-1}-1,
\]
where the coefficients $c_a(A)\in \mathbb F_p$ are uniquely determined by the conditions
\[
\sum_{a\in A}c_a(A)a^m=\begin{cases}
0, \text{ if\ }  0\leq m<\alpha-1\\
1, \text{ if\ } m=\alpha-1.
\end{cases}
\]
Here, we set $a^0=1$ for all $a\in \mathbb F_p$.
\end{definition}

The coefficients $c_a(A)$ always exist due to non-degeneracy of Vandermonde matrix. However, we need a more precise formula for $c_a(A)$.

\begin{lemma}[Explicit formula for coefficients]
The coefficients $c_a(A)$ satisfy
\[
\sum_{a\in A}\frac{c_a(A)}{1-ax}=\frac{x^{\alpha-1}}{\prod\limits_{a\in A}(1-ax)}.
\]
In particular,
\[
c_a(A)=\frac{1}{\prod\limits_{a'\in A\setminus\{a\}}(a-a')}
\]
and for any $t\in \mathbb F_p$ we have $c_{a+t}(A+t)=c_a(A)$, i.e. these coefficients are shift-invariant.
\end{lemma}
\begin{proof}
The existence and uniqueness of $c_a(A)$ follows from non-degeneracy of the Vandermonde matrix. To verify the generating function identity, observe that
\[
\frac{x^{\alpha-1}}{\prod\limits_{a\in A}(1-ax)}=x^{\alpha-1}+O(x^\alpha),
\]
while
\[
\sum_{a\in A}\frac{c_a(A)}{1-ax}=\sum_{m\geq 0}\left(\sum_{a\in A}c_a(A)a^m\right)x^m.
\]
Comparing the first $\alpha$ coefficients yields the result. The explicit formula for $c_a(A)$ is obtained by multiplying both sides by $1-ax$ and evaluating at $x=\frac{1}{a}$. If $a=0$, we should instead notice that the right-hand side is well-defined at infinity and the left-hand side is equal to $c_0(A)$ (this is equivalent to ``setting $x=\infty$'' in the identity). The shift-invariance follows directly from the formula.
\end{proof}
The generating functions in this proof are treated as rational functions and formal power series with coefficients in the field $\mathbb F_p$, i.e. as elements of $\mathbb F_p(x)$ and $\mathbb F_p((x))$.

As a consequence of Lemma 1, we also establish the following result

\begin{lemma}[Complete homogeneous symmetric polynomials]
For any $A$ and $m\geq 0$ we have
\[
\sum_{a\in A}c_a(A)a^{m+\alpha-1}=h_m(A)=\sum_{i_1\leq i_2\leq\ldots\leq i_m}a_{i_1}\ldots a_{i_m}
\]
--- the $m$-th complete homogeneous symmetric polynomial of $A$. Here, $a_1,\ldots,a_\alpha$ is some arbitrary enumeration of elements of $A$.
In particular, for $m=1,2,3$ and $p>3$ we get
\[
h_1(A)=p_1(A), h_2(A)=\frac12(p_1(A)^2+p_2(A)), h_3(A)=\frac16 p_1(A)^3+\frac12 p_2(A)p_1(A)+\frac 13 p_3(A),
\]
where $p_k(A)=\sum\limits_{a\in A}a^k$ is the power sum of $A$.
\end{lemma}
\begin{proof}
This follows from the generating function formula in Lemma 1 and the relation
\[
\frac{1}{1-ax}=1+ax+a^2x^2+a^3x^3+\ldots
\]
The formulas for $h_k$ in terms of $p_k$ follow from Newton's identites. To divide by $6$ we need to assume that $p>3$.
\end{proof}

The results of Lemmas 1 and 2 are relatively standard and can be found in books on enumerative combinatorics. See, for instance, \cite[Chapter 7, Section 7.5 and Exercise 7.4]{Sta}.

Hanson and Petridis \cite{HaPe} proved that if $A,B\subseteq \mathbb F_p$ are subsets such that $A+B\subseteq \mu_d\cup \{0\}$, then $|A||B|\leq d+|(-A)\cap B|$. Their proof shows that $HP(x;A,d)$ has a zero of order at least $|A|-1$ at every $b\in B$ and at least $|A|$ at every $b\in B\setminus (-A)$. This motivates the definition of a $d$-critical pair.

\begin{definition}[$d$-critical pair]
A pair $(A,B)$ of subsets of $\mathbb F_p$ with $|A|,|B|>1$ is $d$-critical if 
\[A+B\subseteq \mu_d\cup \{0\}
\] and 
\[
|A||B|=d+|(-A)\cap B|
\]
\end{definition}

Note that if $A+B$ is equal to $\mu_d$ or $\mu_d\cup\{0\}$, then $(A,B)$ is necessarily $d$-critical. Indeed, by the mentioned result we have $|A||B|\leq d+|(-A)\cap B|$ and the reverse bound follows by a simple counting argument, since $|(-A)\cap B|$ is the number of solutions to $a+b=0$ in $a\in A, b\in B$. Also, in this case for every $c\in \mu_d$ there is a unique pair $(a,b)\in A\times B$ such that $a+b=c$. This proves the following
\begin{lemma}
If $A+B=\mu_d$ or $\mu_d\cup \{0\}$, then for $1\leq k<d$ we have
\[
\sum_{a\in A,b\in B}(a+b)^k=0.
\]
\end{lemma}

Finally, the main result of this section is a factorization of $HP(x;A,d)$ in case of $d$-criticality. For the reader's convenience, we first obtain the result of Hanson and Petridis: if $b+A\subseteq \mu_d\cap \{0\}$, then $HP(x;A,d)$ has a root of order at least $\alpha-1$ at $x=b$. Our proof is slightly different, because we mainly rely on the shift-invariance of $c_a(A)$.

\begin{lemma}
Suppose that the pair $(A,B)$ is $d$-critical and $|A|=\alpha, |B|=\beta$. Define $\varepsilon(b)=1$ if $b\in B\cap (-A)$ and $0$ otherwise. Then there exists $C\in \mathbb F_p$ such that
\[
HP(x;A,d)=C\prod_{b\in B}(x-b)^{\alpha-\varepsilon(b)}
\]
\end{lemma}

\begin{proof}
Indeed, if $b\in B$ then
\[
HP(b;A,d)=\sum_{a\in A}c_a(A)(a+b)^{d+\alpha-1}-1=\sum_{a\in A}c_a(A)(a+b)^{\alpha-1}-1=
\]
\[
=\sum_{a\in A}c_{a+b}(A+b)(a+b)^{\alpha-1}-1=0,
\]
because $(a+b)^{d+1}=a+b$. Similarly, for $0<j\leq \alpha-2$ we have
\[
HP^{(j)}(b;A,d)=f_j\sum_{a\in A}c_a(A)(a+b)^{d+\alpha-1-j}=f_j\sum_{a\in A}c_a(A)(a+b)^{\alpha-1-j}=0
\]
Here $f_j$ is $(d+\alpha-1)\ldots(d+\alpha-j)$. Additionally, if $b$ is not in $-A$, then $(a+b)^d=1$ for all $a\in A$ and $(\alpha-1)$-th derivative is also zero. Therefore, each $b$ is a zero of $HP$ of order $\alpha-\varepsilon(b)$. The degree of $HP$ is equal to $d$ (see \cite[proof of Theorem 1.2]{HaPe}) and by $d$-criticality we have
\[
\sum_{b\in B}(\alpha-\varepsilon(b))=\alpha\beta-|B\cap(-A)|=d,
\]
which concludes the proof: all the factors $(x-b)^{\alpha-\varepsilon(b)}$ appear the factorization of $HP$ and the degrees on both sides match.
\end{proof}

Comparison of coefficients in Lemma 4 yields $d$ relations between symmetric polynomials of $A$ and $B$. Total number of elementary symmetric polynomials of $A$ and $B$ is $\alpha+\beta=O(\sqrt{d})$ hence one expects that there are too many relations and the contradiction is achievable. This is precisely what occurs in the next three sections. 

Throughout this paper, if $s(x)$ and $r(x)$ are rational functions over the field $K$ and $\kappa\in K$, then
\[
s(x)=r(x)+O((x-\kappa)^l)
\]
means that the rational function $s(x)-r(x)$ has a zero of order at least $l$ at $x=\kappa$.

\section{ Lev-Sonn conjecture}

Throughout this section, $|A|=\alpha$. We prove Theorem 1 by establishing a stronger result.

\begin{theorem}
Let $p$ be a prime, $d\mid p-1$ and $1<d<p-1$. If $d\not\in \{2,6\}$, then there is no set $A$ such that $(A,-A)$ is a $d$-critical pair, except when $p=41$ and $d=20$.
\end{theorem}
For $p=41$, the set $A=\{0,1,9,32,40\}$ satisfies $A-A\subset \mu_{20}\cup \{0\}$. This inclusion is strict and our proof implies that this is the only such example up to affine transformation.

Since $(A,-A)$ is $d$-critical and $|(-A)\cap (-A)|=\alpha$, we have $\alpha(\alpha-1)=d$.
From Lemma 4 we derive
\begin{lemma}
If $(A,-A)$ is $d$-critical and $p_1(A)=\sum\limits_{a\in A}a=0$, then either $d\in \{2,6\}$ or $p_2(A)=p_3(A)=0$.
\end{lemma}
\begin{proof}
Notice that $\varepsilon(-a)=1$ for all $a\in A$, so Lemma 4 implies
\[
HP(x;A,d)=C\prod_{a\in A}(x+a)^{\alpha-1}.
\]
Let us compute $C$. On the right, the coefficient of $x^{\alpha(\alpha-1)}=x^d$ is $C$. On the left, it is
\[
\sum_{a\in A}c_a(A){\alpha+d-1\choose d}a^{\alpha-1}={\alpha+d-1 \choose d}.
\]
Therefore, $C={\alpha+d-1 \choose d}$.

Computing the coefficient of $x^{d-2}$, we get on the left
\[
{\alpha+d-1\choose d-2}\sum_{a\in A}c_a(A)a^{\alpha+1}={\alpha+d-1\choose d-2}h_2(A)=\frac12{\alpha+d-1\choose d-2}p_2(A),
\]
because of the assumption $p_1(A)=0$. On the right we get
\[
C\left(\frac{(\alpha-1)(\alpha-2)}{2}\sum_{a\in A}a^2+(\alpha-1)^2\sum_{\substack{a_1,a_2\in A\\a_1\neq a_2}}a_1a_2\right).
\]
Since $p_1(A)=0$, we see that the second sum is $-p_2(A)/2$ and we obtain
\[
\frac12{\alpha+d-1\choose d-2}p_2(A)=-\frac12{\alpha+d-1\choose d}(\alpha-1)p_2(A).
\]
Since $\alpha+d-1=\alpha^2-1<p$, none of the factorials that appear below are divisible by $p$. If $p_2(A)\neq 0$, we have
\[
{\alpha+d-1\choose d-2}\equiv -(\alpha-1){\alpha+d-1\choose d}\pmod p.
\]
Hence
\[
\frac{(\alpha+d-1)!}{(d-2)!(\alpha+1)!}\equiv -\frac{(\alpha-1)(\alpha+d-1)!}{d!(\alpha-1)!}\pmod p.
\]
This leads to
\[
\frac{d!}{(d-2)!}\equiv -\frac{(\alpha-1)(\alpha+1)!}{(\alpha-1)!}\pmod p.
\]
Thus,
\[
d(d-1)\equiv -(\alpha-1)\alpha(\alpha+1)\pmod p.
\]
Since $d=\alpha(\alpha-1)$, we derive
\[
\alpha+1\equiv 1-d \pmod p.
\]
Therefore, $\alpha+d$ is divisible by $p$, so $\alpha^2\equiv 0\pmod p$, which is a contradiction.

Next, let us compute a coefficient at $x^{d-3}$. If $d\neq 2,6$, then $\alpha\geq 4$ and $d\geq 12$, so the $-1$ in the formula for $HP$ does not affect the $(d-3)$-rd coefficient. For the polynomial on the left, the coefficient of $x^{d-3}$ is
\[
{\alpha+d-1\choose d-3}\sum_{a\in A}c_a(A)a^{\alpha+2}={\alpha+d-1\choose d-3}h_3(A)=\frac13{\alpha+d-1\choose d-3}p_3(A),
\]
because $p_1(A)=p_2(A)=0$. On the right we get $C$ times the coefficient of $\prod_{a\in A}(x+a)^{\alpha-1}=x^d\prod_{a\in A}(1+ax^{-1})^{\alpha-1}$. Since $p_1(A)=p_2(A)=0$ imply that the second elementary symmetric polynomial of $A$ is zero,
\[
\prod_{a\in A}(1+ax^{-1})=1+e_3(A)x^{-3}+\ldots,
\]
where $e_3(A)$ is the third elementary symmetric polynomial. But by Newton's identities
\[
e_3(A)=\frac16 p_1(A)^3-\frac12 p_1(A)p_2(A)+\frac13 p_3(A)=\frac13 p_3(A).
\]
Hence the $(d-3)$-rd coefficient of $\prod\limits_{a\in A}(x+a)^{\alpha-1}$  is $\frac{\alpha-1}{3}p_3(A)$. Therefore, if $p_3(A)\neq 0$, we have
\[
{\alpha+d-1\choose d-3}\equiv(\alpha-1)C\equiv(\alpha-1){\alpha+d-1 \choose d} \pmod p
\]
This implies that
\[
\frac{(\alpha+d-1)!}{(d-3)!(\alpha+2)!}\equiv \frac{(\alpha-1)(\alpha+d-1)!}{d!(\alpha-1)!}\pmod p.
\]
Therefore,
\[
\frac{d!}{(d-3)!}\equiv \frac{(\alpha-1)(\alpha+2)!}{(\alpha-1)!} \pmod p,
\]
which reduces to
\[
d(d-1)(d-2)\equiv (\alpha-1)\alpha(\alpha+1)(\alpha+2) \pmod p.
\]
Substituting $d=\alpha^2-\alpha$, expanding the brackets and subtracting, we derive
\[
\alpha^6-3\alpha^5-\alpha^4+3\alpha^3\equiv 0\pmod p.
\]
Luckily, the polynomial on the left factors as $\alpha^3(\alpha+1)(\alpha-1)(\alpha-3)$ and $1<\alpha<p-1$, so $\alpha=3$ and $d=\alpha^2-\alpha=6$, which concludes the proof.
\end{proof}

As a consequence of Lemma 5, we prove
\begin{lemma}
For $d\neq 2,6$ and $d$-critical $(A,-A)$, the identities
\[
p_2(A)=\frac{1}{\alpha}p_1(A)^2, p_3(A)=\frac{1}{\alpha^2}p_1(A)^3
\]
hold.
\end{lemma}
\begin{proof}
Indeed, if $(A,-A)$ is $d$-critical, then for $B=A-\frac{p_1(A)}{\alpha}$ the pair $(B,-B)$ is $d$-critical and $p_1(B)=0$. By Lemma 5, $p_2(B)=p_3(B)=0$. But
\[
p_2(B)=\sum_{a\in A}\left(a-\frac{p_1(A)}{\alpha}\right)^2=p_2(A)-\frac{p_1(A)^2}{\alpha}
\]
and therefore $p_2(A)=\frac{p_1(A)^2}{\alpha}$. Next, we obtain
\[
0=p_3(B)=\sum_{a\in A}\left(a-\frac{p_1(A)}{\alpha}\right)^3=p_3(A)-\frac{3p_1(A)p_2(A)}{\alpha}+\frac{3p_1(A)^3}{\alpha^2}-\frac{p_1(A)^3}{\alpha^2},
\]
which concludes the proof, because the second and the third summands cancel out.
\end{proof}

The next crucial step is to notice that we can produce more $d$-critical pairs from $(A,-A)$.

\begin{lemma}
If $(A,-A)$ is $d$-critical, then for any $a\in A$ the set
\[
A^a=\{0\}\cup \left\{\frac{1}{a-a'}:a'\in A, a'\neq a\right\}.
\]
also forms a $d$-critical pair $(A^a,-A^a)$.
\end{lemma}
\begin{proof}
Obviously, the set $A^a$ has the same size as $A$. So it remains to prove that $A^a-A^a\subseteq \mu_d\cup\{0\}$. Clearly, for every non-zero $x\in A^a$ we have $\pm x\in \mu_d$. Suppose that both $x$ and $y$ in $A^a$ are non-zero, then for some $a',a''$ we have
\[
x-y=\frac{1}{a-a'}-\frac{1}{a-a''}=\frac{a'-a''}{(a-a')(a-a'')}\in \mu_d,
\]
which concludes the proof.
\end{proof}
Lemmas 6 and 7 imply
\begin{corollary}
If $(A,-A)$ is $d$-critical and $d\neq 2$ or $6$, then for any $a\in A$ we have
\begin{equation}
\label{rat2}
\sum_{a'\in A\setminus\{a\}}\frac{1}{(a-a')^2}=\frac{1}{\alpha}\left(\sum_{a'\in A\setminus\{a\}}\frac{1}{a-a'}\right)^2
\end{equation}
\begin{equation}
\label{rat3}
\sum_{a'\in A\setminus\{a\}}\frac{1}{(a-a')^3}=\frac{1}{\alpha^2}\left(\sum_{a'\in A\setminus\{a\}}\frac{1}{a-a'}\right)^3
\end{equation}
\end{corollary}

We are going to establish the impossibility of these relations in a slightly more general setting. 

In the next theorem, $\alpha$ is just the size of $A$ and is not required to satisfy a relation $\alpha^2-\alpha=d$.

\begin{theorem}
Suppose that $K$ is an algebraically closed field. Assume that $A\subseteq K$ is a subset of size $\alpha>1$. If $\mathrm{char}\,K>0$, assume also that $\alpha<\sqrt{\mathrm{char}\,K}$. Suppose that the relations \ref{rat2} and \ref{rat3} hold for all $a\in A$. Then $\alpha=5$ and for some $C,D\in K$ we have
\[
A=\{C,C+D,C-D,C+Di,C-Di\},
\]
where $i\in K$ and $i^2=-1$.
\end{theorem}
\begin{proof}
To achieve this result, we also apply Stepanov's polynomial method, although in a different way. Since the relations $\ref{rat2},\ref{rat3}$ are shift-invariant, we can always assume that the sum of elements of $A$ is zero. Let
\[
f(x)=\prod_{a\in A}(x-a)=\sum_{0\leq n\leq \alpha}A_n x^n.
\]
Let us study the Laurent expansion of $\frac{f'}{f}(x)$ around the point $x=a$. We have
\[
\frac{f'}{f}(x)=\frac{1}{x-a}+\sum_{a'\in A\setminus\{a\}}\frac{1}{x-a'}=\frac{1}{x-a}+\sum_{a'\in A\setminus\{a\}}\frac{1}{a-a'+x-a}=
\]
\[
=\frac{1}{x-a}+\sum_{a'\in A\setminus\{a\}}\left(\frac{1}{a-a'}-\frac{x-a}{(a-a')^2}+\frac{(x-a)^2}{(a-a')^3}\right)+O((x-a)^3).
\]
This implies that
\[
\frac{f'}{f}(x)=\frac{1}{x-a}+\sum_{a'\in A\setminus\{a\}}\frac{1}{a-a'}-(x-a)\sum_{a'\in A\setminus\{a\}}\frac{1}{(a-a')^2}+(x-a)^2\sum_{a'\in A\setminus\{a\}}\frac{1}{(a-a')^3}+O((x-a)^3).
\]
Next, let 
\[s=\frac{1}{\alpha}\sum\limits_{a'\in A\setminus\{a\}}\frac{1}{a-a'}.\]

From the formula above and relations \ref{rat2} and \ref{rat3} we deduce
\[
\frac{f'}{f}(x)=\frac{1}{x-a}+\alpha s-(x-a)\alpha s^2+(x-a)^2\alpha s^3+O((x-a)^3)=\frac{1}{x-a}+\frac{\alpha s}{1+s(x-a)}+O((x-a)^3).
\]
This means that the logarithmic derivative of $f$ can be approximated in the formal neighbourhood of $x=a$ by a much simpler rational function. Unfortunately, due to the dependence of this rational function on $a$ and $s$, these rational functions do not necessarily ``glue together'' for different $a\in A$. So, instead we construct some differential operator, which will annihilate the corresponding approximation (in its non-logarithmic form). To do so, define
\[
F_a(x)=(x-a)(1+s(x-a))^\alpha.
\]
Then we see that
\[
\frac{f'}{f}(x)=\frac{F_a'}{F_a}(x)+O((x-a)^3),
\]
so
\[
\frac{f'F_a-fF_a'}{fF_a}(x)=O((x-a)^3).
\]
Multiplying both sides by $f/F_a$ and using the fact that this is $O(1)$ at $x=a$, we see that
\[
\left(\frac{f}{F_a}\right)'=\frac{f'F_a-fF_a'}{fF_a}(x)=O((x-a)^3).
\]
Integrating, we conclude that there is a constant $c\in F$ such that
\[
\frac{f}{F_a}(x)=c+O((x-a)^4),
\]
so we arrive at a close approximation
\[
f(x)=cF_a(x)+O((x-a)^5).
\]
Next, let us define a differential operator $\mathcal D:K[x]\to K[x]$, designed to annihilate $F_a(x)$ for all $a$:
\[
\mathcal Dg=4\alpha(\alpha-2)g'g'''-3(\alpha-1)(\alpha-2)g''^2-\alpha(\alpha+1)gg''''.
\]
We claim that $\mathcal DF_a=0$. To prove this, notice that for any $m\geq 0$ we have
\[
F_a(x)^{(m)}=(\alpha)_ms^m(x-a)(1+s(x-a))^{\alpha-m}+m(\alpha)_{m-1}s^{m-1}(1+s(x-a))^{\alpha-m+1},
\]
where $(\alpha)_m=\alpha(\alpha-1)\ldots(\alpha-m+1)$ and $(\alpha)_0=1$ (one can also set $(\alpha)_{-1}=\frac{1}{\alpha-1}$, but this exact value is not important, since for $m=0$ we multiply it by $0$). In particular, for $m=0,1,2$ we get
\[
F_a(x)^{(m)}F_a(x)^{(4-m)}=
\]
\[
=s^2(1+s(x-a))^{2\alpha-2}((\alpha)_ms(x-\alpha)+m(\alpha)_{m-1}(1+s(x-a)))((\alpha)_{4-m}s(x-\alpha)+(4-m)(\alpha)_{3-m}(1+s(x-a)))
\]
\[
=s^2(1+T)^{2\alpha-2}((\alpha)_mT+m(\alpha)_{m-1}(T+1))((\alpha)_{4-m}T+(4-m)(\alpha)_{3-m}(T+1)),
\]
where we set $T=s(x-\alpha)$ for convenience. Thus, we derive
\[
\mathcal DF_a(x)=s^2(1+T)^{2\alpha-2}(4\alpha(\alpha-2)(\alpha T+T+1)((\alpha)_3T+3(\alpha)_2(T+1))-3(\alpha-1)(\alpha-2)((\alpha)_2T+2\alpha(T+1))^2-
\]
\[
-\alpha(\alpha+1)T((\alpha)_4T+4(\alpha)_3(T+1)))=s^2(1+T)^{2\alpha-2}G(\alpha,T)
\]
Direct computation then shows $G(\alpha,T)=0$.

Observe now that $\mathcal D$ exhibits a degree of continuity in $(x-a)$-adic topology. In particular, since
\[
f(x)=cF_a(x)+O((x-a)^5),
\]
we have
\[
f^{(m)}(x)=cF_a^{(m)}(x)+O((x-a)^{5-m}).
\]
Hence,
\[
f'f'''=c^2F_a'F_a'''+O((x-a)^2)
\]
and
\[
f''^2=c^2F_a''^2+O((x-a)^3).
\]
Now, both $f$ and $F$ are $O(x-a)$, so we also obtain
\[
ff''''=c^2F_aF_a''''+O((x-a)^2).
\]
Therefore,
\[
\mathcal Df=c^2\mathcal DF+O((x-a)^2)=O((x-a)^2).
\]
In particular, the polynomial $\mathcal Df$ has a zero of order at least $2$ at every $x=a$. The degree of $\mathcal Df$ is clearly at most $2\alpha-4$, since the derivative lowers the degree by $1$, while the number of zeros of this polynomial counted with multiplicity is at least $2\alpha$. Therefore, $\mathcal Df=0$.

Next, let $l$ be the largest power below $\alpha$ such that $A_l\neq 0$, i.e. $f(x)$ has a non-zero coefficient at $x^l$. Then the coefficient of $\mathcal Df$ at $x^{\alpha+l-4}$ is equal to
\[
D(\alpha,l)A_l=(4\alpha(\alpha-2)(\alpha(l)_3+l(\alpha)_3)-6(\alpha-1)(\alpha-2)(\alpha)_2(l)_2-\alpha(\alpha+1)((\alpha)_4+(l)_4))A_l.
\]
The polynomial $D(\alpha,l)$ admits the following factorization:
\[
D(\alpha,l)=-\alpha(\alpha-l+1)(\alpha-l)(\alpha-l-1)(\alpha^2-(l+5)\alpha-l+6).
\]
Since we assumed that the sum of elements of $A$ is zero, $A_{\alpha-1}=0$, so $l<\alpha-1$. On the other hand, if the characteristic of $K$ is equal to $p>0$, then $\alpha<\sqrt{p}$, so none of the first four factors can be zero. Therefore, we necessarily have $\alpha^2-(l+5)\alpha-l+6=0$ as an element of the field $K$. As an integer, this is at least $\alpha^2-(\alpha+3)\alpha-\alpha+8=8-4\alpha>-\alpha^2$ and at most
$\alpha^2-5\alpha+6<\alpha^2$ (since $\alpha>1$). Therefore, even if $\mathrm{char}\,K\neq 0$, we have an equality of integers
\[
\alpha^2-(l+5)\alpha-l+6=0.
\]
In particular, the discriminant with respect to $\alpha$ should be a square. I.e. $(l+5)^2+4l-24=l^2+14l+1=(l+7)^2-48$ must be a square of an integer. A short search among divisors of $48$ shows that $(l,\alpha)\in \{(0,2),(0,3),(1,5),(6,11)\}$. Clearly, $\alpha=2$ or $3$ are impossible: in the first case, we have $A=\{a_1,a_2\}$ and the relation \ref{rat2} says that
\[
\left(\frac{1}{a_1-a_2}\right)^2=\frac{1}{2}\frac{1}{(a_1-a_2)^2},
\]
which is a contradiction. If $\alpha=3$, then $A=\{a_1,a_2,a_3\}$ and for non-zero $x=\frac{1}{a_1-a_2}, y=\frac{1}{a_1-a_3}$ we get by $\ref{rat2},\ref{rat3}$
\[
3(x^2+y^2)=(x+y)^2, 9(x^3+y^3)=(x+y)^3,
\]
hence $3(1+z^2)=(1+z)^2$ and $9(z^3+1)=(1+z)^3$ for $z=y/x$. The resultant of polynomials $3(1+z^2)-(1+z)^2$ and $9(1+z^3)-(1+z)^3$ is equal to $216$, hence $\mathrm{char}\,K$ is either $2$ or $3$, which is a contradiction.

If $\alpha=11, l=6$, then
\[
f(x)=x^{11}+A_6x^6+\ldots+A_0
\]
and we compute
\[
\mathcal Df(x)=-27720\cdot A_5x^{12}- 88704\cdot A_4x^{11} -199584\cdot A_3x^{10} - 380160\cdot A_2x^9 - 
\]
\[
-(5400\cdot A_6^2 + 653400\cdot A_1)x^8 - (7200\cdot A_5A_6+1045440\cdot A_0)x^7+O(x^6)
\]
Next, under our assumptions, in this case characteristic of $K$ is either $0$ or larger than $121$, but all the coefficients above factor into primes $\leq 11$:
\[
27720=2^3\cdot 3\cdot 5\cdot 7\cdot 11, 88704=2^7\cdot 3^2\cdot 7\cdot 11, 199584=2^5\cdot 3^4\cdot 7\cdot 11,
\]
\[
380160=2^8\cdot 3^3\cdot 5\cdot 11, 1045440=2^6\cdot 3^3\cdot 5\cdot 11^2.
\]
This implies that necessarily $A_5=A_4=A_3=A_2=A_0=0$. Since all prime factors of $5400$ are $2,3,5$, from the $8$-th coefficient we also get $A_6^2=121A_1$. This means that our polynomial is of the form
\[
f(x)=x^{11}+11ux^6+u^2x,
\]
where $u=A_6/11$. Multiplying the whole set $A$ by any fifth root of $1/u$, we see that one can take $u=1$ without loss of generality. Although in this case we indeed have $\mathcal Df=0$, we now verify that the set $A$ of roots of $f(x)=x^{11}+11x^6+x$ fails to satisfy condition \ref{rat2}. Indeed, if $\xi$ is a root of $f(x)$, then the Taylor expansion of $f(x)$ gives
\[
f(x)=f'(\xi)(x-\xi)+\frac{f''(\xi)}{2}(x-\xi)^2+\frac{f'''(\xi)}{6}(x-\xi)^3+O((x-\xi)^4)
\]
and taking logarithmic derivatives we see that
\[
\frac{f'}{f}(x)=\frac{1}{x-\xi}+\frac{f''(\xi)}{2f'(\xi)}+\left(\frac{f'''(\xi)}{3f(\xi)}-\frac{f''(\xi)^2}{4f'(\xi)^2}\right)(x-\xi)+O((x-\xi)^2).
\]
In our case, $\alpha=11$ and the condition \ref{rat2} reads
\[
-\frac{f''(\xi)^2}{44f'(\xi)^2}=\frac{f'''(\xi)}{3f(\xi)}-\frac{f''(\xi)^2}{4f'(\xi)^2}.
\]
Equivalently, we must have
\[
15f''(\xi)^2=22f'(\xi)f'''(\xi)
\]
for all roots $\xi$ of $f(x)$. Left-hand side is
\[
15(110\xi^9+330\xi^4)^2=1500\cdot 121\xi^8(\xi^5+3)^2.
\]
The right-hand side simplifies to
\[
22(11\xi^{10}+66\xi^5+1)(990\xi^8+1320\xi^3)=60\cdot 121\xi^3(11\xi^{10}+66\xi^5+1)(3\xi^5+4).
\]
Assuming $\xi\neq 0$, from invertibility of $3,5$ and $11$ in $K$ we obtain,
\[
25(\xi^5+3)^2=(11\xi^{10}+66\xi^5+1)(3\xi^5+4).
\]
All non-zero roots of $f$ satisfy $\xi^{10}=-1-11\xi^5$, so $(\xi^5+3)^2=\xi^{10}+6\xi^5+9=8-5\xi^5$ and $11\xi^{10}+66\xi^5+1=-10-55\xi^5$. Therefore,
\[
(11\xi^{10}+66\xi^5+1)(3\xi^5+4)=-(10+55\xi^5)(3\xi^5+4)=-40-250\xi^5-165\xi^{10}=
\]
\[
=-40-250\xi^5+165+1815\xi^5=125+1565\xi^5.
\]
Subtracting, we deduce that the identity \ref{rat2} for non-zero roots of $f$ is equivalent to
\[
1690\xi^5=75.
\]
This means that $\xi^5$ equals  $\frac{75}{1690}=\frac{15}{338}$, there are at most $5$ such values of $\xi$, but $f(x)$ must have $10$ non-zero roots, which is a contradiction.

Finally, if $\alpha=5$ and $l=1$, we get
\[
f(x)=x^5+ax+b,
\]
and dividing by a fourth root of $-a$ we can assume that $a=-1$. Then $f(x)=x^5-x+b$ and
\[
\mathcal Df=60f'f'''-36f''^2-30ff''''=-3600bx,
\]
hence $b=0$ and up to affine transformation the set $A$ is equal to $\{0,\pm 1,\pm i\}.$
\end{proof}

Now we are ready to finish the proof of Theorem 4.
\begin{proof}[Proof of Theorem 4]
If $(A,-A)$ is $d$-critical, then by Corollary 2 we have either $d\in\{2,6\}$ or the relations $\ref{rat2}$ and $\ref{rat3}$ hold for all $a\in A$. In the latter case Theorem 5 implies that $\alpha=5$ and $A=\{C,C\pm D,C\pm iD\}$ (here an algebraically closed $K$ is $K=\overline{\mathbb F}_p$). Thus, $d=\alpha(\alpha-1)=20$, $C+D-C=D\in\mu_{20}$ and $C+D-(C-iD)=D(1+i)$, hence $D^{20}=(1+i)^{20}D^{20}=1$. Since $(1+i)^2=2i$, we have $(1+i)^{20}=-1024$, so $p$ divides $1025=25\cdot 41$. On the other hand, $20=d\mid p-1$, therefore $p=41$ and we get the described example.
\end{proof}

From this we immediately get the Lev-Sonn conjecture:
\begin{proof}[Proof of Theorem 1]
Assume that $A-A=\mu_d$. Then $(A,-A)$ is $d$-critical and by Theorem 4 we get either $d=2$ or $6$ or $p=41$ and $d=20$. For $p=41$, the set $A$ takes form $\{C,C\pm D,C\pm iD\}$. Due to the identity $(C+D)-(C-iD)=(C+iD)-(C-D)$ we get at most $5\cdot 4-1=19$ distinct non-zero differences, hence  $A-A\neq\mu_{20}$. This concludes the proof.
\end{proof}

\section{The $\alpha=\beta$ theorem.}

In this section, we prove Theorem 2 and make some crucial steps towards the proof of S\'ark\"ozy's conjecture. Let $d$ be as above and assume that $A+B=\mu_d$, where $|A|=\alpha$, $|B|=\beta$ and $\alpha,\beta>1$. By the results of Section 2, the $(A,B)$ is $d$-critical. Moreover, $0\not\in A+B$, hence $\alpha\beta=d$. Next, let $n$ be the least integer $n>0$ such that $p_n(A)\neq0$. Such $n$ necessarily exists, since otherwise all the elementary symmetric polynomials of $A$ are zero, which is impossible for any set of cardinality at least $2$. Next, $n$ is also the least integer with $p_n(B)\neq 0$ and $\beta p_n(A)=-\alpha p_n(B)$. Indeed, $n\leq \alpha<d$, hence by Lemma 3 we have
\[
\sum_{a\in A,b\in B}(a+b)^n=0.
\]
By binomial theorem, we obtain
\[
\sum_{k=0}^n {n \choose k}p_k(A)p_{n-k}(B)=0.
\]
All the summands with $0<k<n$ are zero, because $p_k(A)=0$, and we deduce
\[
\alpha p_n(B)+\beta p_n(A)=0.
\]
If for some $k<n$ we have $p_k(B)=0$, then from
\[
\sum_{a\in A,b\in B}(a+b)^k=0
\]
we derive $\alpha p_k(B)=0$, which is a contradiction.

Now, let $m=m(A)$ be the least $m\leq \alpha$ such that $p_m(A)\neq 0$ and $n$ does not divide $m$. Define $m(B)$ similarly. Then we get
\begin{lemma}
If $m(A)$ or $m(B)$ exist, then $m(A)=m(B)$ and $\alpha p_m(B)+\beta p_m(A)=0$.
\end{lemma}
\begin{proof}
Suppose that $m(A)\neq m(B)$. Without loss of generality, we can assume that $m(B)>m(A)$. Then for $k<m$ we have $p_k(A)\neq 0$ only if $k\mid n$ and $p_k(B)\neq 0$ only if $k\mid n$. Also, $p_m(B)=0$. Since
\[
\sum_{a\in A,b\in B}(a+b)^m=0,
\]
we get
\[
\sum_{l=0}^m {m \choose l}p_l(A)p_{m-l}(B)=0.
\]
The summand for $l=m$ is equal to $\beta p_m(A)\neq 0$. For $0<l<m$ we have either $n\nmid l$ or $n \nmid m-l$, so $p_l(A)p_{m-l}(B)=0$. For $l=0$ we get $\alpha p_m(B)=0$. Therefore, we have $\beta p_m(A)=0$, which is a contradiction. Examining the sum of $(a+b)^m$ once more, we get the second part of the lemma.
\end{proof}

Unfortunately, in this case there is no ``fractional linear trick'', like in Lemma 7, so we need to work with Hanson-Petridis polynomials more.

Since $(A,B)$ is $d$-critical and $B\cap (-A)=\emptyset$, $\varepsilon(b)=0$ for all $b\in B$, thus Lemma 4 yields
\[
HP(x;A,d)=C\prod_{b\in B}(x-b)^{\alpha}
\]
for some $C\in\mathbb F_p$. Choose any $b\in B$. Applying the change of variables $x\mapsto \frac{1}{x}+b$ to the identity above, we derive
\begin{lemma}
For any $b\in B$, let $A_b=\left\{\frac{1}{a+b}, a\in A\right\}$. Then for some constants $C_0,C_\infty$ we have
\[
\sum_{a\in A}c_{\frac{1}{a+b}}(A_b)(a+b)\left(x+\frac{1}{a+b}\right)^{\alpha+d-1}=C_\infty x^{\alpha+d-1}+C_0x^{\alpha-1}\prod_{b'\in B\setminus\{b\}}\left(x+\frac{1}{b-b'}\right)^\alpha
\]
\end{lemma}
\begin{proof}
Substitution $x\mapsto \frac{1}{x}+b$ shows that
\[
HP\left(\frac{1}{x}+b;A,d\right)=C\prod_{b'\in B}\left(\frac{1}{x}+b-b'\right)^{\alpha}
\]
On the right, we get
\[
C_1x^{-\alpha\beta}\prod_{b'\in B\setminus\{b\}}\left(x+\frac{1}{b-b'}\right)^\alpha,
\]
where
\[
C_1=C\prod_{b'\in B\setminus\{b\}}(b-b')^{\alpha}.
\]
On the left, we obtain
\[
HP\left(\frac{1}{x}+b;A,d\right)=\sum_{a\in A}c_a(A)\left(\frac{1}{x}+a+b\right)^{\alpha+d-1}-1.
\]
Multiplying both sides by $x^{\alpha+d-1}$, using $\alpha\beta=d$ and $(a+b)^{\alpha+d-1}=(a+b)^{\alpha-1}$ for all $a\in A, b\in B$, we conclude that
\[
\sum_{a\in A}c_a(A)(a+b)^{\alpha-1}\left(x+\frac{1}{a+b}\right)^{\alpha+d-1}=
\]
\begin{equation}
\label{AplusB}
=x^{\alpha+d-1}+C_1x^{\alpha-1}\prod_{b'\in B\setminus\{b\}}\left(x+\frac{1}{b-b'}\right)^{\alpha}
\end{equation}

Now we need to relate the coefficients $c_a(A)$ to $c_{\frac{1}{a+b}}(A_b)$. To do so, notice that for any $a\in A$
\[
c_{\frac{1}{a+b}}(A_b)=\frac{1}{\prod\limits_{a'\in A\setminus{a}}\left(\frac{1}{a+b}-\frac{1}{a'+b}\right)}
\]
by Lemma 1. But $\frac{1}{a+b}-\frac{1}{a'+b}=\frac{a'-a}{(a+b)(a'+b)}$. Therefore,
\[
c_{\frac{1}{a+b}}(A_b)=\frac{\prod\limits_{a'\in A\setminus\{a\}}(a+b)(a'+b)}{\prod\limits_{a'\in A\setminus{a}}(a'-a)}=C_\infty c_a(A)(a+b)^{\alpha-2},
\]
where
\[
C_\infty=\prod_{a'\in A}(a'+b).
\]
Note: the power of $(a+b)$ is $\alpha-2$, because this factor appears $\alpha-1$ times in the numerator and the same number of times in the product $C_\infty(a+b)^{\alpha-2}$. In conclusion, we have
\[
\sum_{a\in A}c_a(A)(a+b)^{\alpha-1}\left(x+\frac{1}{a+b}\right)^{\alpha+d-1}=C_\infty^{-1}\sum_{a\in A}c_{\frac{1}{a+b}}(A_b)(a+b)\left(x+\frac{1}{a+b}\right)^{\alpha+d-1}.
\]
Multiplying \ref{AplusB} by $C_\infty$, we derive the Lemma 9 with $C_0=C_1C_\infty$.
\end{proof}

Now we use Lemma 9 to obtain two important relations.
\begin{lemma}
For any $b\in B$ we have
\begin{equation}
\label{relx}
\tag{Relation X}
\sum_{a\in A}\frac{1}{a+b}=\frac{(\alpha+1)\alpha}{d-1}\sum_{b'\in B\setminus\{b\}}\frac{1}{b-b'}.    
\end{equation}

\begin{equation}
\label{rely}  
\tag{Relation Y}
\left(\sum_{a\in A}\frac{1}{a+b}\right)^2+\sum_{a\in A}\frac{1}{(a+b)^2}=\frac{\alpha(\alpha+1)(\alpha+2)}{(d-1)(d-2)}\left(\alpha\left(\sum_{b'\in B\setminus\{b\}}\frac{1}{b-b'}\right)^2-\sum_{b'\in B\setminus\{b\}}\frac{1}{(b-b')^2}\right)
\end{equation}
\end{lemma}

The Relation X will be mostly used in this section and Relation Y will only be used in the next one.
\begin{proof}
Due to Lemma 9, we have
\[
\sum_{a\in A}c_{\frac{1}{a+b}}(A_b)(a+b)\left(x+\frac{1}{a+b}\right)^{\alpha+d-1}=C_\infty x^{\alpha+d-1}+C_0x^{\alpha-1}\prod_{b'\in B\setminus\{b\}}\left(x+\frac{1}{b-b'}\right)^\alpha.
\]
The coefficient $C_0$ is equal to the $(d-1)$-th coefficient of the left-hand side. Thus, by binomial formula we get
\[
C_0=\sum_{a\in A}c_{\frac{1}{a+b}}(A_b)(a+b)\left(\frac{1}{a+b}\right)^\alpha{\alpha+d-1\choose \alpha}={\alpha+d-1\choose \alpha}.
\]
To establish Relations X and Y we need two more coefficients. The coefficient of $x^{d-2}$ on the right is
\[
C_0\alpha\sum_{b'\in B\setminus\{b\}}\frac{1}{b-b'}.
\]
On the left, it is
\[
{\alpha+d-1\choose \alpha+1}\sum_{a\in A}c_{\frac{1}{a+b}}(A_b)\left(\frac{1}{a+b}\right)^\alpha.
\]
Applying Lemma 2 for $m=1$, we see that this is
\[
{\alpha+d-1\choose \alpha+1}\sum_{a\in A}\frac{1}{a+b}.
\]
Therefore,
\[
\sum_{a\in A}\frac{1}{a+b}=\frac{\alpha C_0}{{\alpha+d-1\choose \alpha+1}}\sum_{b'\in B\setminus\{b\}}\frac{1}{b-b'}.
\]
To complete the proof of Relation X, notice that
\[
\frac{\alpha C_0}{{\alpha+d-1\choose \alpha+1}}=\frac{\frac{\alpha(\alpha+d-1)!}{\alpha!(d-1)!}}{\frac{(\alpha+d-1)!}{(\alpha+1)!(d-2)!}}=\frac{\alpha(\alpha+1)!(d-2)!}{\alpha!(d-1)!}=\frac{\alpha(\alpha+1)}{d-1}.
\]
The coefficient of $x^{d-3}$ on the right is
\[
C_0\left(\alpha^2\sum_{b'\neq b'', b',b''\in B\setminus\{b\}}\frac{1}{(b-b')(b-b'')}+{\alpha \choose 2}\sum_{b'\in B\setminus\{b\}}\frac{1}{(b-b')^2}\right).
\]
Next, the first sum inside the brackets can be expressed as
\[
\sum_{b'\neq b'', b',b''\in B\setminus\{b\}}\frac{1}{(b-b')(b-b'')}=\frac12\left(\left(\sum_{b\in B\setminus\{b\}}\frac{1}{b-b'}\right)^2-\sum_{b\in B\setminus\{b\}}\frac{1}{(b-b')^2}\right).
\]
On the other hand, $-\frac{\alpha^2}{2}+{\alpha \choose 2}=-\frac{\alpha}{2}$, so collecting the terms we get
\[
\frac{C_0\alpha}{2}\left(\alpha\left(\sum_{b'\in B\setminus\{b\}}\frac{1}{b-b'}\right)^2-\sum_{b'\in B\setminus\{b\}}\frac{1}{(b-b')^2}\right).
\]
Then, on the left the coefficient at $x^{d-3}$ is
\[
{\alpha+d-1\choose \alpha+2}\sum_{a\in A}c_{\frac{1}{a+b}}(A_b)\left(\frac{1}{a+b}\right)^{\alpha+1}.
\]
Lemma 2 now gives
\[
\sum_{a\in A}c_{\frac{1}{a+b}}(A_b)\left(\frac{1}{a+b}\right)^{\alpha+1}=h_2(A_b)=\frac12\left(\left(\sum_{a\in A}\frac{1}{a+b}\right)^2+\sum_{a\in A}\frac{1}{(a+b)^2}\right).
\]
Comparing, we get
\[
\left(\sum_{a\in A}\frac{1}{a+b}\right)^2+\sum_{a\in A}\frac{1}{(a+b)^2}=\frac{C_0\alpha}{{\alpha+d-1\choose \alpha+2}}\left(\alpha\left(\sum_{b'\in B\setminus\{b\}}\frac{1}{b-b'}\right)^2-\sum_{b'\in B\setminus\{b\}}\frac{1}{(b-b')^2}\right).
\]
Next, we see that
\[
\frac{C_0\alpha}{{\alpha+d-1\choose \alpha+2}}=\frac{\frac{\alpha(\alpha+d-1)!}{\alpha!(d-1)!}}{\frac{(\alpha+d-1)!}{(\alpha+2)!(d-3)!}}=\frac{\alpha(\alpha+2)!(d-3)!}{\alpha!(d-1)!}=\frac{\alpha(\alpha+1)(\alpha+2)}{(d-1)(d-2)},
\]
which concludes the proof.
\end{proof}

Before the application of Lemma 10, we need to introduce the notion of differential forms and their residues. Some further formulas are obtained using differential forms. They also can be established combinatorially, but this approach is somewhat easier for our goals.

\begin{definition}
Let $K$ be a field. A differential form on $\mathbb P^1(K)$ is an expression $f(x)dx$, where $f(x)$ is a rational function. For $b\neq \infty$ the residue
\[
\mathrm{Res}_{x=b}\,f(x)dx
\]
is the coefficient of $\frac{1}{x-b}$ in the Laurent expansion of $f(x)$ at $x=b$. The residue at infinity
\[
\mathrm{Res}_{x=\infty}f(x)dx
\]
is $-a_{1}$ if $f(x)$ has a Laurent expansion at infinity of the form
\[
f(x)=\sum_{n\geq -N}a_n x^{-n}.
\]
\end{definition}
The following reciprocity law is well known and admits vast generalizations (see, for example, \cite{Tat}).
\begin{lemma}[Sum of residues formula]
Let $K$ be an algebraically closed field and $\omega$ be a differential form on $\mathbb P^1(K)$. Then
\[
\sum_{y\in \mathbb P^1(K)}\mathrm{Res}_{x=y}\omega=0.
\]
\end{lemma}

Notice that the sum above always contains only finitely many terms. Here we apply the Relation X of Lemmas 10 and 11 to prove the $\alpha=\beta$ theorem. First, we need a few computations.
\begin{lemma}
Suppose that $k$ is equal to either $n$ or $m$ (see the beginning of this section). Then
\[
\sum_{b\in B}b^{k+1}\sum_{b'\in B\setminus\{b\}}\frac{1}{b-b'}=p_k(B)\left(\beta-\frac{k+1}{2}\right)
\]
and
\[
\sum_{b\in B}b^{k+1}\sum_{a\in A}\frac{1}{a+b}+(-1)^k\sum_{a\in A}a^{k+1}\sum_{b\in B}\frac{1}{a+b}=\alpha p_k(B)+(-1)^k\beta p_k(A).
\]
\end{lemma}
From this point on, we set
\[
g(x)=\prod_{b\in B}(x-b)
\]
and
\[
h(x)=\prod_{a\in A}(x+a).
\]
\begin{proof}[Proof of Lemma 12]
Define
\[
\omega_{2,0}^k=x^{k+1}\left(\frac{g'}{g}(x)\right)^2dx.
\]
Clearly, the only singluarities of $\omega_{2,0}^k$ are located at $x=b$ with $b\in B$ and infinity. Let us compute the residues. First, the infinity. We get
\[
\frac{g'}{g}(x)=\sum_{b\in B}\frac{1}{x-b}=\sum_{b\in B}\frac{x^{-1}}{1-bx^{-1}}=
\]
\[
=\sum_{l\geq 0}p_l(B)x^{-l-1}.
\]
Therefore,
\[
x^{k+1}\left(\frac{g'}{g}(x)\right)^2=\sum_{r,s\geq 0}p_r(B)p_s(B)x^{k-r-s-1}.
\]
We get
\[
\mathrm{Res}_{x=\infty}\omega_{2,0}^k=-\sum_{r+s=k}p_r(B)p_s(B).
\]
If $k=n$, the only summands with non-zero $p_r(B)$ are $p_0(B)p_k(B)$ and $p_k(B)p_0(B)$, therefore,
\[
\mathrm{Res}_{x=\infty}\omega_{2,0}^k=-2\beta p_k(B).
\]
For $k=m$, if $r+s=m$, then either $r$ or $s$ is not divisible by $n$, so we get the exact same formula.

Now, take $b\in B$. Around $b$ we get
\[
\frac{g'}{g}(x)=\frac{1}{x-b}+\sum_{b'\in B\setminus\{b\}}\frac{1}{x-b'}=\frac{1}{x-b}+\sum_{b'\in B\setminus\{b\}}\frac{1}{x-b+b-b'}=
\]
\[
=\frac{1}{x-b}+\sum_{b'\in B\setminus\{b\}}\frac{1}{b-b'}-(x-b)\sum_{b'\in B\setminus\{b\}}\frac{1}{(b-b')^2}+O((x-b)^2).
\]
Hence
\[
\left(\frac{g'}{g}(x)\right)^2=\frac{1}{(x-b)^2}+\frac{2}{x-b}\sum_{b'\in B\setminus\{b\}}\frac{1}{b-b'}+O(1).
\]
On the other hand,
\[
x^{k+1}=b^{k+1}+(k+1)b^k(x-b)+O((x-b)^2),
\]
so we get
\[
\mathrm{Res}_{x=b}\omega_{2,0}^k=(k+1)b^k+2b^{k+1}\sum_{b'\in B\setminus\{b\}}\frac{1}{b-b'}.
\]
Now, by sum of residues formula we get
\[
\sum_{b\in B}\left((k+1)b^k+2b^{k+1}\sum_{b'\in B\setminus\{b\}}\frac{1}{b-b'}\right)=-\mathrm{Res}_{x=\infty}\omega_{2,0}^k=2\beta p_k(B).
\]
Subtracting $(k+1)p_k(B)$ and dividing by $2$, we get the first relation of Lemma 12.

To establish the second relation, set
\[
\omega_{1,1}^k=\frac{g'}{g}(x)\frac{h'}{h}(x)dx.
\]
The singularities are now located at $B,-A$ and $\infty$. First, let us compute the residue at $\infty$. We have
\[
\frac{g'}{g}(x)=\sum_{l\geq 0}p_l(B)x^{-l-1}
\]
and similarly
\[
\frac{h'}{h}(x)=\sum_{l\geq 0}(-1)^lp_l(A)x^{-l-1}.
\]
Therefore,
\[
x^{k+1}\frac{g'}{g}(x)\frac{f'}{f}(x)=\sum_{r,s\geq 0}p_s(B)p_r(A)(-1)^rx^{k-r-s-1},
\]
so
\[
\mathrm{Res}_{x=\infty}\omega_{1,1}^k=-\sum_{r+s=k}p_s(B)p_r(A)(-1)^r.
\]
Arguing as before, we see that 
\[
\mathrm{Res}_{x=\infty}\omega_{1,1}^k=-\alpha p_k(B)-(-1)^k\beta p_k(A).
\]
Next, for $b\in B$ we obtain
\[
\mathrm{Res}_{x=b}\omega_{1,1}^k=b^{k+1}\frac{h'}{h}(b)=\sum_{a\in A}\frac{b^{k+1}}{a+b}
\]
and for $a\in A$
\[
\mathrm{Res}_{x=-a}\omega_{1,1}^k=(-1)^{k+1}a^{k+1}\sum_{b\in B}\frac{1}{-a-b}=(-1)^k\sum_{b\in B}\frac{a^{k+1}}{a+b}.
\]
Applying sum of residues formula, we arrive at the conclusion of this lemma.
\end{proof}

We are now in position to prove Theorem 2.
\begin{proof}[Proof of Theorem 2]
If $A+B=\mu_d$, then by Relation X from Lemma 10 we get
\[
\sum_{a\in A}\frac{1}{a+b}=\frac{(\alpha+1)\alpha}{d-1}\sum_{b'\in B\setminus\{b\}}\frac{1}{b-b'}
\]
Let $k$ be either $n$ or $m$. Multiplying both sides by $b^{k+1}$ and summing, we obtain
\[
\sum_{b\in B}\sum_{a\in A}\frac{b^{k+1}}{a+b}=\frac{(\alpha+1)\alpha}{d-1}\sum_{b\in B}b^{k+1}\sum_{b'\in B\setminus\{b\}}\frac{1}{b-b'}.
\]
Lemma 12 shows that the right-hand side of this formula is equal to $\frac{(\alpha+1)\alpha}{d-1}p_k(B)\left(\beta-\frac{k+1}{2}\right)$. Switching $A$ and $B$ in Lemma 12, we prove similarly,
\[
\sum_{a\in A}\sum_{b\in B}\frac{a^{k+1}}{a+b}=\frac{(\beta+1)\beta}{d-1}p_k(A)\left(\alpha-\frac{k+1}{2}\right).
\]
Multiplying this by $(-1)^k$ and summing, we get from the second equality of Lemma 12
\[
\alpha p_k(B)+(-1)^k\beta p_k(A)=\sum_{b\in B}b^{k+1}\sum_{a\in A}\frac{1}{a+b}+(-1)^k\sum_{a\in A}a^{k+1}\sum_{b\in B}\frac{1}{a+b}=
\]
\[
=\frac{(\alpha+1)\alpha}{d-1}p_k(B)\left(\beta-\frac{k+1}{2}\right)+(-1)^k\frac{(\beta+1)\beta}{d-1}p_k(A)\left(\alpha-\frac{k+1}{2}\right)
\]
By Lemma 8, we know that $\alpha p_k(B)+\beta p_k(A)=0$. Suppose first that $k$ is even, then on the left we get $0$ and on the right, using $\beta p_k(A)=-\alpha p_k(B)$ we get
\[
\frac{\alpha p_k(B)}{d-1}\left((\alpha+1)\left(\beta-\frac{k+1}{2}\right)-(\beta+1)\left(\alpha-\frac{k+1}{2}\right)\right).
\]
The first factor is non-zero in $\mathbb F_p$, hence
\[
(\beta-\alpha)\frac{k+3}{2}=(\alpha+1)\left(\beta-\frac{k+1}{2}\right)-(\beta+1)\left(\alpha-\frac{k+1}{2}\right)\equiv 0\pmod p
\]
and we get either $\alpha=\beta$ and the proof is complete, or $k+3\equiv 0\pmod p$. This would imply that $\frac{p-1}{2}\geq d\geq k\geq p-3$, so $p\leq 5$. If $p=2$ or $3$, then there are no choices for $d$ and if $p=5$, then $d=2$ and we necessarily have $\alpha=\beta=2$ (which is also impossible, since $\alpha\beta=d$).

Now, if $k$ is odd, then the left-hand side is $\alpha p_k(B)-\beta p_k(A)=2\alpha p_k(B)$. On the right, we obtain
\[
\frac{\alpha p_k(B)}{d-1}\left((\alpha+1)\left(\beta-\frac{k+1}{2}\right)+(\beta+1)\left(\alpha-\frac{k+1}{2}\right)\right).
\]
Dividing both sides by $\frac{\alpha p_k(B)}{d-1}$, we get
\[
2d-2\equiv 2\alpha\beta-(\alpha+\beta)\frac{k-1}{2}-k-1 \pmod p.
\]
Since $\alpha\beta=d$, we derive
\[
(k-1)(\alpha+\beta+2)\equiv 0\pmod p.
\]
Therefore, we must have $k=1$: otherwise, $\alpha+\beta\geq p-2$, but $\alpha+\beta\leq d\leq (p-1)/2$, which is a contradiction, because $p>3$. In particular, $n=1$ and $m$ does not exist. On the other hand, we can always shift $A$ by $t=-\frac{p_1(A)}{\alpha}$ and $B$ by $-t$ to get $p_1(A)=p_1(B)=0$, which will yield $n>1$ and a contradiction. This concludes the proof of the $\alpha=\beta$ theorem.
\end{proof}

From our proof we also derive one more useful result:
\begin{corollary}
If $p_1(A)=p_1(B)=0$, then $n$ and $m$ are even.
\end{corollary}
In subsequent calculations we always assume $n>1$.

\section{Quadratic residues are additively irreducible}

In this section, we use Relation Y of Lemma 10 to prove S\'ark\"ozy's conjecture. Our goal is to establish the following quadratic relation between $\alpha$ and $k$, where $k=n$ or $m$. All the proofs use $\alpha=\beta$.
\begin{lemma}
If $k=n$ or $m$, then
\[
2(3k-2)(k-1)\alpha+(k+2)(k+3)\equiv 0\pmod p.
\]
\end{lemma}
To derive this relation, we first compute a few more sums similar to Lemma 12. Certain elements of $\mathbb F_p$ appear in the calculations and we denote them by $\gamma_0,\gamma_1,\ldots$
These gammas will depend on $\alpha$ and $k\in \{n,m\}$. We assume that $p\geq 19$, because Theorem 2 implies $p=2\alpha^2+1$ with $\alpha\geq 2$. In particular, $\alpha$ must be divisible by $3$, otherwise we would have $3\mid p$. Therefore, $\alpha\geq 3$ and $p\geq 2\cdot 3^2+1=19$.

Assume that $A+B=\mu_{\frac{p-1}{2}}$ for a prime $p$. Then $d=\frac{p-1}{2}\equiv -\frac12\pmod p$. Lemma 10 gives for any $b\in B$
\[
\sum_{a\in A}\frac{1}{a+b}=\gamma_0\sum_{b'\in B\setminus\{b\}}\frac{1}{b-b'}
\]
and
\[
\left(\sum_{a\in A}\frac{1}{a+b}\right)^2+\sum_{a\in A}\frac{1}{(a+b)^2}=\gamma_1\left(\alpha\left(\sum_{b'\in B\setminus\{b\}}\frac{1}{b-b'}\right)^2-\sum_{b'\in B\setminus\{b\}}\frac{1}{(b-b')^2}\right)
\]
with
\begin{equation}
\label{gamma0}
    \gamma_0=\frac{\alpha(\alpha+1)}{d-1}\equiv -\frac{2\alpha(\alpha+1)}{3} \pmod p.
\end{equation}

and
\begin{equation}
\label{gamma1}
    \gamma_1=\frac{\alpha(\alpha+1)(\alpha+2)}{(d-1)(d-2)}\equiv \frac{4\alpha(\alpha+1)(\alpha+2)}{15}\pmod p.
\end{equation}

We now need to compute residues for three more differential forms.
\begin{lemma}
For $k=n$ or $m$ we have
\[
\sum_{b\in B}b^{k+2}\left(\left(\sum_{b'\in B\setminus\{b\}}\frac{1}{b-b'}\right)^2-\sum_{b'\in B\setminus\{b\}}\frac{1}{(b-b')^2}\right)=\gamma_2p_k(B),
\]
where
\begin{equation}
\label{gamma2}
    \gamma_2=\gamma_2(k)=\alpha^2-(k+2)\alpha+\frac{(k+1)(k+2)}{3}.
\end{equation}

\end{lemma}
\begin{proof}
Consider a differential form
\[
\omega_{3,0}^k=x^{k+2}\left(\frac{g'}{g}(x)\right)^3dx.
\]
Its poles are located at $B$ and $\infty$. First, let us compute the residue at $\infty$.

We have
\[
\frac{g'}{g}(x)=\sum_{l\geq 0}p_l(B)x^{-l-1},
\]
so
\[
\mathrm{Res}_{x=\infty}\omega_{3,0}^k=-\sum_{r+s+t=k}p_r(B)p_s(B)p_t(B).
\]
Arguing as in Lemma 14, we get
\[
\mathrm{Res}_{x=\infty}\omega_{3,0}^k=-3\alpha^2 p_k(B).
\]
On the other hand, if $b\in B$, then
\[
\frac{g'}{g}(x)=\frac{1}{x-b}+\sum_{b'\in B\setminus\{b\}}\frac{1}{b-b'}-(x-b)\sum_{b'\in B\setminus\{b\}}\frac{1}{(b-b')^2}+O((x-b)^2).
\]
Cubing, we see that
\[
\left(\frac{g'}{g}(x)\right)^3=\frac{1}{(x-b)^3}+\frac{3}{(x-b)^2}\sum_{b'\in B\setminus\{b\}}\frac{1}{b-b'}+\frac{3}{x-b}\left(\left(\sum_{b'\in B\setminus\{b\}}\frac{1}{b-b'}\right)^2-\sum_{b'\in B\setminus\{b\}}\frac{1}{(b-b')^2}\right)+O(1)
\]
Since
\[
x^{k+2}=b^{k+2}+(k+2)b^{k+1}(x-b)+\frac{(k+2)(k+1)}{2}b^k(x-b)^2+O((x-b)^3),
\]
we obtain
\[
\mathrm{Res}_{x=b}\omega_{3,0}^k=\frac{(k+2)(k+1)}{2}b^k+3(k+2)b^{k+1}\sum_{b'\in B\setminus\{b\}}\frac{1}{b-b'}+
\]
\[
+3b^{k+2}\left(\left(\sum_{b'\in B\setminus\{b\}}\frac{1}{b-b'}\right)^2-\sum_{b'\in B\setminus\{b\}}\frac{1}{(b-b')^2}\right).
\]
Applying sum of residues formula, we get
\[
\frac{(k+2)(k+1)}{2}p_k(B)+3(k+2)\sum_{b\in B}b^{k+1}\sum_{b'\in B\setminus\{b\}}\frac{1}{b-b'}+
\]
\[
+3\sum_{b\in B}b^{k+2}\left(\left(\sum_{b'\in B\setminus\{b\}}\frac{1}{b-b'}\right)^2-\sum_{b'\in B\setminus\{b\}}\frac{1}{(b-b')^2}\right)=3\alpha^2 p_k(B).
\]
By the first formula of Lemma 12, we have
\[
3(k+2)\sum_{b\in B}b^{k+1}\sum_{b'\in B\setminus\{b\}}\frac{1}{b-b'}=p_k(B)\left(3\alpha(k+2)-\frac{3(k+1)(k+2)}{2}\right).
\]
Moving the first and the second summand to the right and dividing by $3$, we arrive at the conclusion of Lemma 14.
\end{proof}

Denote also
\begin{equation}
\label{gamma3}
\gamma_3=\alpha-\frac{k+1}{2},
\end{equation}
then Lemma 12 states that
\[
\sum_{b\in B}b^{k+1}\sum_{b'\in B\setminus\{b\}}\frac{1}{b-b'}=\gamma_3 p_k(B).
\]

In the next Lemma, we deal with a differential form involving a second derivative.
\begin{lemma}
For $k=n$ or $m$ we have
\[
\sum_{a\in A,b\in B}\frac{b^{k+2}-a^{k+2}}{(a+b)^2}=\gamma_4p_k(B),
\]
where
\begin{equation}
\label{gamma4}
\gamma_4=(k+2)\gamma_0\gamma_3-k\alpha.
\end{equation}
\end{lemma}
\begin{proof}
Let
\[
\psi=x^{k+2}\left(\frac{g'}{g}(x)\right)'\frac{h'}{h}(x)dx.
\]
Its poles are located at $B$, $-A$ and $\infty$. First, we compute the residue at $\infty$.

We have
\[
\frac{h'}{h}(x)=\sum_{l\geq 0}(-1)^lp_l(A)x^{-l-1}
\]
and
\[
\left(\frac{g'}{g}(x)\right)'=-\sum_{l\geq 0}(l+1)p_l(B)x^{-l-2}.
\]
Therefore,
\[
\mathrm{Res}_{x=\infty}\psi=\sum_{r+s=k}(-1)^r(s+1)p_r(A)p_s(B).
\]
This sum evaluates to
\[
(-1)^kp_k(A)p_0(B)+(k+1)p_0(A)p_k(B).
\]
Note that Corollary 2 implies $(-1)^k=1$. Also, since $\alpha=\beta$ we have $p_k(A)=-p_k(B),$ so we derive
\[
\mathrm{Res}_{x=\infty}\psi=k\alpha p_k(B).
\]
Next, for $a\in A$ we have
\[
\mathrm{Res}_{x=-a}\psi=(-a)^{k+2}\left(\frac{g'}{g}\right)'(-a)=-a^{k+2}\sum_{b\in B}\frac{1}{(a+b)^2}.
\]
Finally, for $b\in B$ we have
\[
\left(\frac{g'}{g}\right)'(x)=-\frac{1}{(x-b)^2}+O(1),
\]
so
\[
\mathrm{Res}_{x=b}\psi=-\frac{\partial}{\partial x}(x^{k+2}\frac{h'}{h}(x))\vert_{x=b}=-(k+2)b^{k+1}\sum_{a\in A}\frac{1}{a+b}+b^{k+2}\sum_{a\in A}\frac{1}{(a+b)^2}.
\]
Applying sum of residues formula, we see that
\[
\sum_{a\in A, b\in B}\frac{b^{k+2}-a^{k+2}}{(a+b)^2}=(k+2)\sum_{b\in B}b^{k+1}\sum_{a\in A}\frac{1}{a+b}-k\alpha p_k(B).
\]
Relation X of Lemma 10 together with the first relation of Lemma 12 concludes the proof.
\end{proof}
The third and the final relation with differential forms involves a square of logarithmic derivative.
\begin{lemma}
For $k=n$ or $m$ we have
\[
\sum_{a\in A}a^{k+2}\left(\sum_{b\in B}\frac{1}{a+b}\right)^2+\frac{2}{\gamma_0}\sum_{b\in B}b^{k+2}\left(\sum_{a\in A}\frac{1}{a+b}\right)^2-\sum_{b\in B,a\in A}\frac{b^{k+2}}{(a+b)^2}=\gamma_5 p_k(B),
\]
where
\begin{equation}
\label{gamma5}
\gamma_5=\alpha^2-(k+2)\gamma_0\gamma_3.
\end{equation}
\end{lemma}
\begin{proof}
Consider a form
\[
\omega_{2,1}^k=x^{k+2}\left(\frac{g'}{g}(x)\right)^2\frac{h'}{h}(x)dx.
\]
Similarly to the previous computations, at $\infty$ we get
\[
\mathrm{Res}_{x=\infty}\omega_{2,1}^k=-\sum_{r+s+t=k}p_r(B)p_s(B)(-1)^tp_t(A).
\]
The only non-zero summands in the sum above correspond to $(0,0,k), (0,k,0)$ and $(k,0,0)$ hence we get
\[
\mathrm{Res}_{x=\infty}\omega_{2,1}^k=-2\alpha^2 p_k(B)-(-1)^k \alpha^2 p_k(A)=-\alpha^2 p_k(B),
\]
because $p_k(A)=-p_k(B)$ and $k$ is even.

Next, at $x=-a$ the residue is
\[
\mathrm{Res}_{x=-a}\omega_{2,1}^k=(-a)^{k+2}\left(\frac{g'}{g}(-a)\right)^2=a^{k+2}\left(\sum_{b\in B}\frac{1}{a+b}\right)^2.
\]
The expansion around $x=b$ gives
\[
\left(\frac{g'}{g}(x)\right)^2=\frac{1}{(x-b)^2}+\frac{2}{x-b}\sum_{b'\in B\setminus\{b\}}\frac{1}{b-b'}+O(1),
\]
so
\[
\mathrm{Res}_{x=b}\omega_{2,1}^k=2b^{k+2}\sum_{b'\in B\setminus\{b\}}\frac{1}{b-b'}\sum_{a\in A}\frac{1}{a+b}+(k+2)b^{k+1}\sum_{a\in A}\frac{1}{a+b}-b^{k+2}\sum_{a\in A}\frac{1}{(a+b)^2}.
\]
Using the relation X, we see that
\[
\sum_{b'\in B\setminus\{b\}}\frac{1}{b-b'}\sum_{a\in A}\frac{1}{a+b}=\frac{1}{\gamma_0}\left(\sum_{a\in A}\frac{1}{a+b}\right)^2.
\]
From the formula
\[
\sum_{b\in B}b^{k+1}\sum_{a\in A}\frac{1}{a+b}=\gamma_0\gamma_3p_k(B)
\]
and sum of residues formula we get the stated result.
\end{proof}

Note that $\alpha$ and $\alpha+1$ are invertible $\mod p$, so the expression $\frac{1}{\gamma_0}$ makes sense. We now use Lemmas 14, 15, 16 and Relation Y of Lemma 10 to prove Lemma 13.
\begin{proof}[Proof of Lemma 13:]
Multiplying Relation Y by $b^{k+2}$ and summing over $B$, we get
\[
\sum_{b\in B}b^{k+2}\left(\sum_{a\in A}\frac{1}{a+b}\right)^2+\sum_{a\in A,b\in B}\frac{b^{k+2}}{(a+b)^2}=
\]
\[=\gamma_1\left(\alpha\sum_{b\in B}b^{k+2}\left(\sum_{b'\in B\setminus\{b\}}\frac{1}{b-b'}\right)^2-\sum_{b\in B}b^{k+2}\sum_{b'\in B\setminus\{b\}}\frac{1}{(b-b')^2}\right)
\]
Lemma 14 implies that

\[
\sum_{b\in B}b^{k+2}\sum_{b'\in B\setminus\{b\}}\frac{1}{(b-b')^2}=\sum_{b\in B}b^{k+2}\left(\sum_{b'\in B\setminus\{b\}}\frac{1}{b-b'}\right)^2-\gamma_2p_2(B).
\]
Therefore, on the right we get
\[
\gamma_1(\alpha-1)\sum_{b\in B}b^{k+2}\left(\sum_{b'\in B\setminus\{b\}}\frac{1}{b-b'}\right)^2+\gamma_1\gamma_2 p_k(B).
\]
Relation X implies that
\[
\sum_{b\in B}b^{k+2}\left(\sum_{b'\in B\setminus\{b\}}\frac{1}{b-b'}\right)^2=\frac{1}{\gamma_0^2}\sum_{b\in B}b^{k+2}\left(\sum_{a\in A}\frac{1}{a+b}\right)^2, 
\]
which results in the formula
\[
\left(1-\frac{\gamma_1(\alpha-1)}{\gamma_0^2}\right)\sum_{b\in B}b^{k+2}\left(\sum_{a\in A}\frac{1}{a+b}\right)^2+\sum_{a\in A,b\in B}\frac{b^{k+2}}{(a+b)^2}=\gamma_1\gamma_2 p_k(B).
\]
In this formula, we can swap $A$ and $B$. Then, using $p_k(A)=-p_k(B)$ and subtracting, we get
\[
\left(1-\frac{\gamma_1(\alpha-1)}{\gamma_0^2}\right)\left(\sum_{b\in B}b^{k+2}\left(\sum_{a\in A}\frac{1}{a+b}\right)^2-\sum_{a\in A}a^{k+2}\left(\sum_{b\in B}\frac{1}{a+b}\right)^2\right)+
\]
\[
+\sum_{a\in A,b\in B}\frac{b^{k+2}-a^{k+2}}{(a+b)^2}=2\gamma_1\gamma_2p_k(B)
\]
The second summand is evaluated in Lemma 15, consequently,
\[
\left(1-\frac{\gamma_1(\alpha-1)}{\gamma_0^2}\right)\left(\sum_{b\in B}b^{k+2}\left(\sum_{a\in A}\frac{1}{a+b}\right)^2-\sum_{a\in A}a^{k+2}\left(\sum_{b\in B}\frac{1}{a+b}\right)^2\right)=(2\gamma_1\gamma_2-\gamma_4)p_k(B)
\]
Next, rewriting Lemma 16 with swapped sets and subtracting, we see that
\[
\left(\frac{2}{\gamma_0}-1\right)\left(\sum_{b\in B}b^{k+2}\left(\sum_{a\in A}\frac{1}{a+b}\right)^2-\sum_{a\in A}a^{k+2}\left(\sum_{b\in B}\frac{1}{a+b}\right)^2\right)=(2\gamma_5+\gamma_4)p_k(B).
\]
Substituting in the above formula and dividing by $p_k(B)$, we finally arrive at relation between gammas, namely
\[
\left(1-\frac{\gamma_1(\alpha-1)}{\gamma_0^2}\right)\left(\frac{2}{\gamma_0}-1\right)^{-1}(2\gamma_5+\gamma_4)=2\gamma_1\gamma_2-\gamma_4.
\]
The quantity $\frac{2}{\gamma_0}-1=-\frac{\alpha^2+\alpha+3}{\alpha(\alpha+1)}$ is invertible, because numerator is at most $\frac{p}{2}+\sqrt{\frac{p}{2}}+3<p$.

All the involved expressions are rational functions in $\alpha$ and $k$ and we have $\alpha^2=d\equiv -\frac{1}{2}\pmod p$, so everything here can be expressed in the form $u\alpha+v$ with $u,v\in \mathbb Q(k)$. This can be done if we work in the ring $\mathbb Q[\alpha]/(2\alpha^2+1)$.

Direct computation then shows that
\[
\left(1-\frac{\gamma_1(\alpha-1)}{\gamma_0^2}\right)\left(\frac{2}{\gamma_0}-1\right)^{-1}(2\gamma_5+\gamma_4)\equiv \left(\frac{2}{15}k^2 + \frac{14}{15}k + \frac{8}{15}\right)\alpha -\frac{1}{15}k^2 - \frac{1}{15}k + \frac{8}{15} \pmod p
\]
and
\[
2\gamma_1\gamma_2-\gamma_4\equiv \left(-\frac{1}{15}k^2 + \frac{19}{15}k + \frac{2}{5}\right)\alpha-\frac{1}{10}k^2 - \frac{7}{30}k + \frac{1}{3} \pmod p.
\]
Subtracting and multiplying by $30$, we arrive at the congruence
\[
(6k^2-10k+4)\alpha+(k^2+5k+6)\equiv 0 \pmod p,
\]
which is the congruence stated in Lemma 13.

Let us now explain this computation a bit more. First of all,
\[
\frac{2}{\gamma_0}-1=-\frac{\alpha^2+\alpha+3}{\alpha(\alpha+1)}\equiv -\frac{5/2+\alpha}{-1/2+\alpha} \pmod p,
\]
hence
\[
\left(\frac{2}{\gamma_0}-1\right)^{-1}\equiv -\frac{\alpha-1/2}{5/2+\alpha}=-\frac{(-1/2+\alpha)(5/2-\alpha)}{25/4-\alpha^2}\equiv -\frac{-5/4+3\alpha-\alpha^2}{25/4+1/2} \pmod p.
\]
Therefore, we get
\[
\left(\frac{2}{\gamma_0}-1\right)^{-1}\equiv -\frac{4}{9}\alpha+\frac{1}{9} \pmod p.
\]
Similarly,
\[
\frac{1}{\gamma_0} \equiv -\frac{3}{2\alpha-1}\equiv -\frac{3(2\alpha+1)}{4\alpha^2-1}\equiv 2\alpha+1 \pmod p.
\]
The rest is a polynomial computation not involving any division by non-constant polynomials. For example,
\[
\left(1-\frac{\gamma_1(\alpha-1)}{\gamma_0^2}\right)\left(\frac{2}{\gamma_0}-1\right)^{-1}\equiv 
\]
\[\equiv\left(1-\frac{4(\alpha-1)\alpha(\alpha+1)(\alpha+2)(2\alpha+1)^2}{15}\right)\left(-\frac49\alpha+\frac19\right) \pmod p.
\]
The last expression equals
\[
\frac{64}{135}\alpha^7 + \frac{176}{135}\alpha^6 + \frac{32}{135}\alpha^5 - \frac{4}{3}\alpha^4 - \frac{104}{135}\alpha^3 + \frac{4}{135}\alpha^2 - \frac{52}{135}\alpha + \frac{1}{9}\equiv -\frac25\pmod p.
\]
Here one can clearly see that all the odd powers of $\alpha$ cancel out.
Remaining computations are carried out similarly.
\end{proof}

As a consequence of Lemma 13 we prove
\begin{lemma}
If $\alpha>9$, then the number $m$ does not exist.
\end{lemma}
\begin{proof}
Indeed, if $m$ existed, we would have two congruences
\[
2(3n-2)(n-1)\alpha+(n+2)(n+3)\equiv 0 \pmod p,
\]
\[
2(3m-2)(m-1)\alpha+(m+2)(m+3)\equiv 0 \pmod p.
\]
Subtracting and dividing by $n-m$, we get
\[
(6n+6m-10)\alpha+n+m+5\equiv 0 \pmod p,
\]
hence
\[
n+m\equiv \frac{-5+10\alpha}{6\alpha+1}\equiv \frac{40\alpha+25}{19} \pmod p,
\]
which is impossible. Indeed, in this case we would get
\[
19n+19m\equiv 40\alpha+25\pmod p,
\]
both numbers are positive and less than $p$. This is because the first $\alpha>9$ with prime value of $p=2\alpha^2+1$ is $\alpha=21$, hence $\alpha\geq 21$ and $p\geq 883$, therefore $40\alpha+25\leq 40\sqrt{p/2}+25<p$. So we must have $19n+19m=40\alpha+25$, which is not true, because the left-hand side is at most $38\alpha$.
\end{proof}

Now we are in position to finish the proof of S\'ark\"ozy's conjecture.
\begin{proof}[Proof of S\'ark\"ozy's conjecture]
First we assume that $\alpha>9$. Since $0\not\in A+B$, at least one of these sets does not contain zero. Without loss of generality, assume that $0\not\in B$. Then the $\alpha-$th elementary symmetric polynomial of $B$ is non-zero (it is a product of all elements of $b$). But $p_k(B)\neq 0$ implies $n\mid k$ by previous lemma and $e_{\alpha}(B)$ can be expressed as a polynomial in $p_k(B)$ of homogeneous degree $\alpha$, therefore we have $n\mid \alpha$. On the other hand, Lemma 13 gives
\[
2(3n-2)(n-1)\alpha+(n+2)(n+3)\equiv 0\pmod p.
\]
Multiplying both sides by $6\alpha-1$ and expanding the brackets, we obtain
\[
(40n+32)\alpha+25n-19n^2-18\equiv 0 \pmod p
\]
We have $\alpha^2<p/2$ and $n\leq \alpha$. The function
\[
(40n+32)\alpha+25n-19n^2-18
\]
increases in $n$ from $32\alpha-18$ to $21\alpha^2+57\alpha-18$. Since $n$ is even, this number is also even, so it takes form $2Mp=2M(2\alpha^2+1)$. On the other hand, for $\alpha\geq 21$ we have
\[
\frac{21\alpha^2+57\alpha-18}{2\alpha^2+1}=10.5+\frac{57\alpha-28.5}{2\alpha^2+1}<12.
\]
This means that $1\leq M\leq 5$. On the other hand, $\alpha\equiv 0\pmod 3$, thus
\[
n-n^2\equiv (40n+32)\alpha+25n-19n^2-18=2M(2\alpha^2+1)\equiv 2M \pmod 3.
\]
This means that either $M=3$ and $n\not\equiv 2\pmod 3$ or $M=2$ or $5$ and $n\equiv 2\pmod 3$. In either case,
\[
n\mid(40n+32)\alpha+25n-19n^2-4M\alpha^2=2M+18.
\]
This observation reduces Theorem 3 to a finite number of cases.
If $M=2$, then we get $n\mid 22$, $n$ is even and $n\equiv 2\pmod 3$, so $n=2$ and
\[
112\alpha-44=(40n+32)\alpha+25n-19n^2-18=8\alpha^2+4,
\]
which has no integer solutions.

If $M=5$, then $n\mid 28$ and either $n=2$ or $n=14$. In the first case, we get
\[
112\alpha-44=20\alpha^2+10,
\]
which has no integer solutions. In the second case, the equation reads
\[
592\alpha-3392=20\alpha^2+10
\]
and we once again get no integer solutions, because the left-hand side is divisible by $4$, but the right-hand side is not.

Finally, if $M=3$, then $n\mid 24$ and $n$ can be $4,6,12$ or $24$ (we omit other divisors of $24$ because of the conditions $2\mid n$ and $n\not\equiv 2\pmod 3$). In this case, we have
\[
(40n+32)\alpha+25n-19n^2-18=12\alpha^2+6.
\]
Treating this as a quadratic equation in $\alpha$, we see that the discriminant \[D(n)=(40n+32)^2-4\cdot12\cdot(19n^2-25n+24)=688n^2+3760n-128\]
must be a square. Substituting our potential values of $n$, we see that 
\[
D(4)/16=1620, D(6)/16=2950, D(12)/16=9004, D(24)/16=30400,
\]
all of these values are non-squares. This concludes the proof for $\alpha>9$. The remaining values are $\alpha=3,6,9$, they correspond to $p=19,73,163$. It is easy to rule out $p=73$ and $p=163$: the polynomial
\[
2(3n-2)(n-1)\alpha+(n+2)(n+3)
\]
from Lemma 13 has no roots in $\mathbb F_{73}$ and its least positive root modulo $163$ is $61$, which is larger than $9=\alpha$. 

The only remaining case is $p=19$, for which $n=\alpha=3$ is one of the roots. Therefore, in this case we have $p_1(A)=p_1(B)=p_2(A)=p_2(B)=0$ and $p_3(A)=-p_3(B)=3S$ are non-zero. The generating polynomial of $A$ is then $x^3-S$ and for $B$ it is $x^3+S$, thus we have $B=-A$, which is a contradiction.

This finally concludes the proof.
\end{proof}

\section{Conclusion}

In this paper, we employed a version of Stepanov's polynomial method, developed by Hanson and Petridis, to prove that proper subgroups of $\mathbb F_p^*$ are not difference sets $(A-A)\setminus\{0\}$, except for the subgroups of orders $2$ and $6$. We also resolved S\'ark\"ozy's conjecture on quadratic residues: for all primes $p$ the set $\mathcal R_p$ of quadratic residues modulo $p$ cannot be expressed as $A+B=\mathcal R_p$ with $|A|,|B|>1$. Computations we encountered are related to congruences and non-congruences between binomial coefficients modulo $p$. This motivates several new problems.

\begin{problem}
Characterize sets $A\subset K$ in a given field $K$ satisfying
\[
\sum_{a'\in A\setminus\{a\}}\frac{1}{(a-a')^2}=\frac{1}{|A|}\left(\sum_{a'\in A\{a\}}\frac{1}{a-a'}\right)^2
\]
for all $a\in A$.
\end{problem}

Theorem 5 from Section 3 shows that if $A$ also satisfies the analogous cubic relation, then $A$ is either a singleton, or $A=\{C,C\pm D, C\pm D\sqrt{-1}\}$, or $|A|>\sqrt{\mathrm{char}\, K}>0$. This motivates the above question.
Note also that the conditions of Problem 1 are equivalent to $|A|$ polynomial equations in $|A|$ variables.

Further study of Hanson-Petridis polynomials, defined in Section 2, allows one to obtain a multiplicative property, satisfied by sets $A$ such that all non-zero differences of elements of $A$ are $d$-th roots of unity in $\mathbb F_p$ and $|A|(|A|-1)=d$.

\begin{problem}
Classify subsets $A\subset K$ with $|A|=\alpha>1$ such that
\[
\prod_{a'\in A\setminus\{a\}}(a-a')^\alpha=-1
\]
for all $a\in A$.
\end{problem}

The set $A=\{0,1,9,32,41\}$ in $\mathbb F_{41}$, described in Section 3, satisfies this property.

Direct comparison of coefficients of Hanson-Petridis polynomials with corresponding products (see Section 3) also motivates the following problem.

\begin{problem}
Let us call $p$ a Lev-Sonn prime, if $p=2\alpha(\alpha-1)+1$ and for some $1<n\leq \alpha$ the congruence
\[
{\alpha^2-1\choose n-1+\alpha}\equiv (-1)^{n-1}{\alpha^2-1 \choose \alpha} \pmod p
\]
holds. Are there infinitely many Lev-Sonn primes?
\end{problem}

A quick search among the first $586$  primes of the form $2\alpha(\alpha-1)+1$ (i.e. $\alpha\leq 3000$) produced only two such primes $p=13$ and $p=41$, corresponding to $(\alpha,n)=(3,3)$ and $(5,5)$.
\section{Acknowledgements}

This work is supported by the Ministry of Science and Higher Education of the Russian Federation (agreement no. 075-15-2022-265) and the Foundation for Advancement of Theoretical Physics and Mathematics BASIS.

The author is grateful to M.A. Korolev, S.V. Konyagin, I.D. Shkredov, P.A. Kucheriaviy, V.F. Lev, M. Rudnev and C. H. Yip for fruitful discussions and suggestions. The author is especially thankful to V.F. Lev for telling him about the Lev-Sonn conjecture.

 \bibliographystyle{amsplain}

\end{document}